\documentclass{amsart}
\usepackage{mathtools,amsfonts,amsmath,amsthm,amssymb,color}
\usepackage{fullpage}
\usepackage{hyperref}
\newtheorem{thm}{Theorem}[section]
\newtheorem{defn}[thm]{Definition}
\newtheorem{prop}[thm]{Proposition}
\newtheorem{conj}[thm]{Conjecture}
\newtheorem{lem}[thm]{Lemma}
\newtheorem{cor}[thm]{Corollary}
\newtheorem{quest}[thm]{Question}

\theoremstyle{remark}
\newtheorem{rmk}[thm]{Remark}
\newtheorem{ex}[thm]{Example}

\newcommand{\B}[1]{\overline{#1}}
\newcommand{\C}{\mathbb{C}}
\newcommand{\R}{\mathbb{R}}
\newcommand{\rb}{\rangle}
\newcommand{\lb}{\langle}
\newcommand{\f}[1]{\mathfrak{#1}}
\newcommand{\bgm}[2]{\lb#1,#2\rb}

\newcommand{\HCFp}{\operatorname{HCF}_+}
\newcommand{\id}{\operatorname{Id}}
\newcommand{\tr}{\operatorname{tr}}
\newcommand{\ad}{\operatorname{ad}}

\title[$\HCFp$ on $\mathsf{SL}_n(\C)$ and perfect solitons]{Positive Hermitian Curvature Flow on special linear groups and perfect solitons}
\author{James Stanfield}

\begin{document}
	\maketitle
	\begin{abstract}
		We study invariant solutions to the Positive Hermitian Curvature Flow, introduced by Ustinovskiy, on complex Lie groups. We show in particular that the canonical scale-static metrics on the special linear groups, arising from the Killing form, are dynamically unstable. This disproves a conjecture of Ustinovskiy. We also construct certain perfect Lie groups that admit at least two distinct invariant solitons for the flow, only one of which is algebraic. This is the second known example of a geometric flow with non-algebraic, homogeneous solitons. The first being the G2-Laplacian flow.
	\end{abstract}
	\section{Introduction}
	\label{sec:intro}
	Suppose $(M,J,g)$ is a Hermitian manifold. Let $\Omega$ and $T$ be respectively the curvature and torsion of the Chern connection. The \emph{Positive Hermitian Curvature Flow} ($\HCFp$), as introduced by Ustinovskiy in \cite{UstPosHCF2019}, is a family of Hermitian metrics satisfying the initial value problem:
	\begin{equation}
		\label{eqn:hcfp}
	\partial_t g_t = -\Theta(g_t);\qquad g_0 = g,
	\end{equation}
	where
	\[
	\Theta(g)_{i\B j} := g^{k \B l}\Omega_{k \B l i\B j } - Q(T)_{i\B j},
	\]
	and
	\begin{equation}
		\label{eqn:Q}
	Q(T)_{i\B j} := \frac{1}{2}g_{k\B l}g^{r\B s}T_{ir}^k\B{T_{js}^l}.
	\end{equation}
	We call $\Theta(g)$ the \emph{torsion-twisted Chern--Ricci curvature}. The flow (\ref{eqn:hcfp}) is a member of a family of \emph{Hermitian Curvature Flows} (HCFs), where $Q(T)$ is taken to be an arbitrary tensor quadratic in $T$, introduced by Streets and Tian in \cite{StreetsAndTianHCF2011}, generalising the K\"ahler--Ricci flow to non-K\"ahler Hermitian manifolds(see \cite[~\S 3]{ustinovskiyThesis} and references therein). The particular choice of torsion term (\ref{eqn:Q}) was made by Ustinovskiy so that the flow (\ref{eqn:hcfp}) preserves curvature positivity conditions on the Chern connection, generalising a similar and very important property of the Ricci Flow that has proven essential in many of its applications (cf. \cite{BohmWilking2Pos2008,BrendleSchoenQuart2009,KRS,HamiltonThreeManifolds1982,HamiltonFourManifolds1986}). In particular, if $g_0$ is \emph{Griffiths non-negative} (resp. positive), i.e. $\Omega(X,JX,JY,Y) \geq 0$ (resp. $>0$) for all $X,Y\in \Gamma(TM)$, then $g_t$ is also Griffiths non-negative (resp. positive) while a solution exists \cite[~Theorem 0.1]{UstHCFHom2017}. Hence the name \emph{Positive} Hermitian Curvature Flow.
	
	
	
	
	It is hoped that (\ref{eqn:hcfp}) can be used as a tool to prove uniformisation results under positive curvature conditions in Hermitian geometry. For instance, the \emph{weak Campana-Peternell} conjecture proposed by Ustinovskiy:
	
	\begin{conj}[{\cite[~Conjecture 1.3]{UstStruct2020}}]
		\label{conj:wCP}
		Let $(M,g)$ be a compact Hermitian manifold with Griffiths non-negative Chern curvature and positive first Chern--Ricci curvature. Then $M$ is isomorphic to a rational homogeneous manifold $\mathsf{P}\backslash \mathsf{G}$, where $\mathsf{G}$ is a complex semisimple Lie group and $\mathsf{P}\leq \mathsf{G}$ is a parabolic subgroup.
	\end{conj}

	
	The $\HCFp$ has shown to be very fruitful in studying Hermitian manifolds satisfying similar assumptions to Conjecture \ref{conj:wCP}. In \cite{UstStruct2020}, Ustinovskiy studied the structure of $(M,g)$ when the assumption on the first Chern--Ricci curvature is dropped. They also proved in \cite{UstPosHCF2019} that if additionally, the Chern curvature is Griffiths positive at some point in $M$, then $M$ is biholomorphic to $\mathbb{CP}^n$ \cite[Propositon 0.3]{UstPosHCF2019}. The proof of the latter statement combines regularising properties of the $\HCFp$ with Mori's solution to Hartshorne's conjecture \cite{MoriProj1979}. A proof of this or Conjecture \ref{conj:wCP} using the asymptotic behaviour of the $\HCFp$ akin to \cite{KRS} in the K\"ahler setting would be interesting.
	
	

	To gain insight into a potential proof of Conjecture \ref{conj:wCP} via the $\HCFp$, it is important to study the flow on manifolds that already have the expected symmetries. To that end, suppose now that $M = \mathsf{H}\backslash \mathsf{G}$ is a \emph{complex homogeneous manifold} (i.e. $\mathsf{G}$ is a complex Lie group and $\mathsf{H} \leq \mathsf{G}$ is a closed complex subgroup). In \cite{UstHCFHom2017}, Ustinovskiy showed that the $\HCFp$ preserves $\mathcal{M}^{\textrm{sub}}(M)$, the space of (typically not $\mathsf{G}$-invariant) submersion metrics on $M$ induced by left-invariant Hermitian metrics on $\mathsf{G}$. Moreover, they showed that an $\HCFp$ solution in $\mathcal{M}^{\textrm{sub}}(M)$ is determined by the corresponding left-invariant solution on $\mathsf{G}$.

	The simplest $\HCFp$ solutions come from \emph{$\HCFp$-static} metrics. A Hermitian metric $g$ is called $\HCFp$-static if $\Theta(g) = \lambda g$ for some $\lambda \in \R$. These are the equivalent of Einstein metrics in the Ricci flow setting. The only known non-K\"ahler examples exist on complex homogeneous manifolds of the form $\mathsf{H}\backslash\mathsf{G}$, where $\mathsf{G}$ is a semisimple complex Lie group. In this case, the metric is given by $\check g_{\operatorname{can}}$, where $g_{\operatorname{can}}$ is a canonical left-invariant metric on $\mathsf{G}$ arising from the \emph{Killing form} $\mathsf{G}$ (see Definition \ref{def:canonicalMetric} and Propositions \ref{prop:canonicalMetric} and \ref{prop:canonicalMetricStatic} for details).
	


	In \cite{UstHCFHom2017}, working at the level of Lie algebras, Ustinovskiy made the following conjecture about left-invariant $\HCFp$ solutions on simple complex Lie groups:
	\begin{conj}[{\cite[~Conjecture 5.4]{UstHCFHom2017}}]
		\label{conj:Ust}
		Let $\mathsf{G}$ be a simple complex Lie group. Then after some rescaling, any left-invariant solution to (\ref{eqn:hcfp}) on $\mathsf{G}$ converges in the $C^\infty$ topology to $\phi^*g_{\operatorname{can}}$ for some inner automorphism $\phi$ of $\mathsf{G}$.
	\end{conj}
	
	Let ${\mathsf{SL}_{n}(\C)}$ be the complex special linear group consisting of complex $n\times n$ matrices of unit determinant. Ustinovskiy observed Conjecture \ref{conj:Ust} holds on ${\mathsf{SL}_{2}(\C)}$ \cite[~Example 5.2]{UstHCFHom2017}. Our main result is a counterexample to Conjecture \ref{conj:Ust} on $\mathsf{SL}_n(\C)$ for $n\geq 3$. We show in particular that $g_{\operatorname{can}}$ is dynamically unstable. Let $\mathsf{H}_{2n+1}(\C)$ denote the complex Heisenberg group of dimension $2n+1$ (see Example \ref{ex:heisenberg}). 
	\begin{thm}
		\label{thm:SLnUnstable}
		Let $n\geq 2$. Then on $\mathsf{SL}_{n+1}(\C)$, there exists a left-invariant solution $(g_t)_{t\in[0,T_{\max})}$ to (\ref{eqn:hcfp}) with $g_0$ chosen arbitrarily close to $g_{\operatorname{can}}$ in the $C^\infty$ topology such that as $t \to T_{\max}$, $(G,J,(T_{\max}-t)^{-1}g_t)$ converges in the Cheeger-Gromov sense to $(\B G,\B J,\B g)$, which is locally isometric to a shrinking $\HCFp$ soliton on $\mathsf{SL}_n(\C)\ltimes \mathsf{H}_{2n+1}(\C)$.
	\end{thm}
%
	
	By convergence in the Cheeger--Gromov sense, we mean for any sequence of times approaching $T_{\max}$, there exists a subsequence $\{t_k\}_{k=1}^
	\infty$, an exhaustion $\{e \in U_k \subset \B {\mathsf{G}}\}_{k=1}^\infty$ of $\B{ \mathsf{G}}$, and biholomorphisms $\varphi_k\colon U_k \to \varphi_k(U_k)\subset{\mathsf{SL}_{n+1}(\C)}$ satisfying $\varphi_k(e) = e$ and such that $\{(T_{\max}-t_k)^{-1}\varphi_k^*g_{t_k}\}_{k=1}^\infty$ converges to $\bar g$ in the $C^\infty$ topology. An $\HCFp$ soliton is a Hermitian manifold $(M,J,g)$ admitting a solution to (\ref{eqn:hcfp}) with $g_0 = g$ that evolves only by scaling and pull-back by biholomorphisms. A soliton is called shrinking (resp. expanding, steady) if the scaling factor is decreasing (resp. increasing, constant).
	
	Considering Theorem \ref{thm:SLnUnstable} in the context of Conjecture \ref{conj:wCP}, the following question is natural:
	\begin{quest}
	Consider a parabolic subgroup $\mathsf{P}\leq {\mathsf{SL}_{n+1}(\C)}$ and a left-invariant $\HCFp$ solution $(g_t)_{t\in T_{\max}}$ coming from Theorem \ref{thm:SLnUnstable}. After a suitable rescaling, is the Cheeger--Gromov limit of $(\mathsf{P}\backslash \mathsf{G},\check g_t)$ as $t\to T_{\max}$ a rational homogeneous manifold?
	\end{quest}

	If the answer is negative, as Theorem \ref{thm:SLnUnstable} suggests, a proof of Conjecture \ref{conj:wCP} via the $\HCFp$ would have to be subtle. Indeed one could not then expect a general Hermitian manifold satisfying the conditions of the conjecture to be deformed into one that is rational homogeneous.
	
	To prove Theorem \ref{thm:SLnUnstable}, we consider a three-parameter family of left-invariant and $\operatorname{Ad}(\mathsf{SU}(n))$-invariant Hermitian metrics on ${\mathsf{SL}_{n+1}(\C)}$, which is preserved by the $\HCFp$. The metric $g_{\operatorname{can}}$ is $\operatorname{Ad}(\mathsf{SU}(n+1))$-invariant, and so is a member of this family. After rescaling, the $\HCFp$ in this ansatz is equivalent to an ODE on $\R_{>0}^2$, and one can show that $g_{\operatorname{can}}$ is unstable. Note that the analysis is more subtle than simply computing the linearisation around $g_{\operatorname{can}}$, since we show that the escaping solutions actually never return. To show the convergence result, we make use of the technique of varying brackets first introduced by Lauret for almost-Hermitian Lie groups in \cite{lauret2015CurvatureFlows}. Specifically, we pull-back a rescaled $\operatorname{Ad}(\mathsf{SU}(n))$-invariant solution to a curve of Lie brackets in $\f{sl}_{n+1}(\C) \otimes \Lambda^2 \f {sl}_{n+1}(\C)^*$, which we show converges to the Lie bracket of a certain semi-direct product $\f{sl}_n(\C) \ltimes \f h_{2n+1}(\C)$. Cheeger--Gromov convergence follows by a technical result of Lauret \cite{lauretConvHomMfds}. 
	
	
	Our second result concerns the uniqueness and properties of left-invariant $\HCFp$ solitons. A complex Lie group with left-invariant Hermitian metric $(\mathsf{G},g)$ is called a \emph{semi-algebraic} $\HCFp$ soliton if it is a soliton in the usual sense, and the biholomorphisms driving the evolution are also Lie group automorphisms of $\mathsf{G}$. A semi-algebraic soliton is called \emph{algebraic} if $\Theta(g)_e = \lambda g_e + g_e(D\cdot,\cdot)$ for some $\lambda \in \R$ and derivation $D \in \operatorname{Der}(\f g = \operatorname{Lie}(\mathsf{G}))$. Equivalent definitions of solitons for other geometric flows can be defined in the obvious way.
	
	It is known that all homogeneous solitons for the Ricci flow are in fact algebraic \cite{jabRicSolAlg14}, but the same question remains open for other flows in general. The author showed in \cite{stanfield21} that left-invariant $\HCFp$-solitons are always algebraic on nilpotent or almost-abelian complex Lie groups, as well as uniqueness up to homothety in the latter case. Moreover in \cite{pujia2020positive}, Pujia proved the same is true on two-step nilpotent Lie groups. Our second main result shows that neither of these properties hold for $\HCFp$-solitons in general.
	\begin{thm}
		\label{thm:TwoSolitons}
		There exist perfect complex Lie groups $\mathsf{G}$ admitting at least two left-invariant, semi-algebraic $\HCFp$-solitons that are distinct up to homothety, only one of which is algebraic.
	\end{thm}
	Recall that a Lie group $\mathsf{G}$ is called \emph{perfect} if its Lie algebra satisfies $[\f g,\f g] = \f g$. To the author's knowledge, this distinguishes the $\HCFp$ as the second known example of a geometric flow with homogeneous solitons that are semi-algebraic, but not algebraic. The first is the $\mathsf{G}_2$-Laplacian flow which was studied in the homogeneous case by Lauret in \cite{lauretLaplacian}.
	
	The solitons in Theorem \ref{thm:TwoSolitons} lie in a particular $\R^\times \ltimes \R$ orbit of left-invariant Hermitian metrics on $\mathsf{G} = \mathsf{H} \ltimes \C^{\dim_{\C} \mathsf{H}}$, where $\mathsf{H}$ is any simply-connected complex Lie group admitting a left-invariant $\HCFp$-static metric (e.g. any simply-connected and semisimple complex Lie group). See section \ref{sec:perfSol} for details on the construction. In particular, note that if $\mathsf{H}$ is unimodular, then so is $\mathsf{G}$.
	
	Theorem \ref{thm:TwoSolitons} therefore demonstrates that the theory of the $\HCFp$ differs substantially to that of others in the HCF family. In \cite{HCFUni2019}, Lafuente, Pujia, and Vezzoni studied the first version of the HCF introduced by Streets and Tian in \cite{StreetsAndTianHCF2011} as the gradient flow of a particular functional. They showed (in contrast to Theorem \ref{thm:TwoSolitons}) that on complex unimodular Lie groups, left-invariant semi-algebraic solitons are algebraic and are unique up to homothety.
	
	Several members of the HCF family are actively being studied in the homogeneous case. We refer the reader to \cite{arrLafHomPlu19,bolingPCFLocHom,pluriclosedTwoStep2015,fino2021pluriclosed,HCFUni2019,PanelliPodestaHCFCmpctHom2020,Pediconi2020,PujiaSolitons2019,pujia2020positive,stanfield21,UstHCFHom2017}.
	
	
	The rest of the article is structured as follows: In section \ref{sec:prelims}, we discuss preliminaries, including $\HCFp$-static metrics and solitons. In section \ref{sec:SLn}, we set up and prove Theorem \ref{thm:SLnUnstable}. Finally, Section \ref{sec:perfSol} contains the proof of Theorem \ref{thm:TwoSolitons}.
	
	\subsubsection*{Acknowledgements}
	I am immensely grateful to my advisor Ramiro Lafuente for his guidance and feedback throughout this project. I also wish to thank Timothy Buttsworth for helpful comments on a final draft of this article. This work was supported by an Australian Government Research Training Program (RTP) Scholarship.
	\section{Preliminaries}
		\label{sec:prelims}
	\subsection{Notation}
	Throughout this article, $(\mathsf{G},J)$ denotes a complex Lie group. That is, $\mathsf{G}$ is some even dimensional Lie group and $J$ is a bi-invariant complex structure on $\mathsf{G}$. Denote by $e \in \mathsf{G}$ the identity and $\f g \cong T_e\mathsf{G}$ the Lie algebra of $\mathsf{G}$ .

	\subsection{Positive Hermitian Curvature Flow}
	Let $(\mathsf{G}^{2n},J,g)$ be a simply-connected complex Lie group with left-invariant Hermitian metric $g$. Then biholomorphism invariance of $\Theta$ implies that it is also left-invariant, thus determined by a bilinear form on the Lie algebra $\f g$ of $\mathsf{G}$. Let $\f g^{1,0} := \f g \otimes \C \cap T^{1,0}_e\mathsf{G}$ and suppose $\{Z_i\}_{i=1}^n \subset \f g^{1,0}$ is a left-invariant $g$-unitary frame. Then by \cite[~Proposition 1.1]{HCFUni2019},
	\begin{equation}
		\label{eqn:Theta}
	\Theta(g)(Z_i,Z_{\B j}) = \frac{1}{2}g(Z_i,[Z_{\B k},Z_{\B l}])\cdot g([Z_k,Z_l],Z_{\B j}),
	\end{equation}
	where we sum over repeated indices. Note that since $(\mathsf{G},J)$ is a complex Lie group, $[Z_i,Z_{\B j}] = 0$.
	
	Since $\Theta(g)$ is biholomorphism invariant, a left-invariant $\HCFp$ solution on $(G,J)$ is equivalent to an ODE of Hermitian inner products on its Lie algebra $(\f g,J_e)$. Hence, short-time existence and uniqueness among left-invariant solutions immediate.
	\subsection{Varying brackets}
	To examine solitons and the limiting behaviour of the $\HCFp$ on Lie groups, we will utilize the technique of varying brackets. This idea was introduced in \cite{lauret2015CurvatureFlows} for almost-Hermitian Lie groups.
	
	Let $(\mathsf{G},J,g)$ be a simply-connected complex Lie group with left-invariant Hermitian metric $g$. This is determined uniquely by the infinitesimal data $(\f g , [\cdot , \cdot],J, \bgm{\cdot}{\cdot})$, where $\f g$ is a $2n$-dimensional vector space, $J$ is a linear complex structure on $\f g$, $[\cdot,\cdot]$ is a Lie bracket on $\f g$ satisfying $[J\cdot,\cdot] = J[\cdot,\cdot]$ and $\bgm{\cdot}{\cdot}$ is a $J$-Hermitian inner product on $\f g$. The Lie group
	\[
	\mathsf{GL}(\f g,J) := \{A \in \mathsf{GL}(\f g): [J,A] = 0\} \cong \mathsf{GL}_n(\C),
	\]
	acts naturally on two spaces associated to $(\f g, J)$. The first is
	\[
	\mathcal{C}(\f g,J) := \{\mu \in \f g \otimes \Lambda^2 \f g^* \textrm{ satisfying the Jacobi identity and } \mu(J\cdot,\cdot) = J\mu\},
	\]
	the \emph{variety of complex Lie brackets} on $\f g$. The second is
	\[
	\operatorname{Sym}^2_+(\f g,J):= \{\textrm{symmetric and positive definite }k \in \f g^* \otimes \f g^* : k(J\cdot,J\cdot) = k\},
	\]
	the space of Hermitian inner products on $(\f g,J)$. The actions are given respectively by
	\[
	h\cdot \mu := h\mu(h^{-1}\cdot,h^{-1}\cdot), \qquad h\cdot k := k(h^{-1}\cdot,h^{-1}\cdot);\qquad \mu\in \mathcal{C}(\f g,J),k \in \operatorname{Sym}^2_+(\f g,J).
	\]
	With this set up, the map $h\colon (\f g,[\cdot,\cdot],J,h^{-1}\cdot\bgm{\cdot}{\cdot}) \to (\f g,h \cdot[\cdot,\cdot],J,\bgm{\cdot}{\cdot})$ is an isomorphism of complex, metric Lie algebras. Since the action of $\mathsf{GL}(\f g,J)$ on $\operatorname{Sym}^2_+(\f g,J)$ is transitive, any complex metric Lie algebra $(\f g,[\cdot,\cdot],J,g)$ is equivalent to $(\f g,h \cdot[\cdot,\cdot],J,\bgm{\cdot}{\cdot})$ for some $h \in \mathsf{GL}(\f g,J)$ both as a complex Lie algebra and $J$-Hermitian inner product space. Thus, the space of $J$-Hermitian inner products on $(\f g,[\cdot,\cdot],J)$ - and hence the space of left-invariant Hermitian metrics on the corresponding simply-connected complex Lie group $(\mathsf{G},J)$- is parametrised by the orbit $\mathsf{GL}(\f g,J)\cdot [\cdot,\cdot]\subset \mathcal{C}(\f g ,J)$.
	
	We now wish to pull-back $\HCFp$ solutions to the orbit $\mathsf{GL}(\f g,J)\cdot [\cdot,\cdot]$ using the above correspondence. Let 
	\[
	\f{gl}(\f g,J) := \operatorname{Lie}(\mathsf{GL}(\f g,J)) = \{A \in \operatorname{End}(\f g) : [A,J] = 0\}
	\] be the Lie algebra of $\mathsf{GL}(\f g,J)$. Define 
	\[
	\mathsf{U}(\f g,J,\bgm{\cdot}{\cdot}) := \{h \in \mathsf{GL}(\f g,J) : h\cdot\bgm{\cdot}{\cdot} = \bgm{\cdot}{\cdot}\},
	\] the unitary group associated to the $J$-Hermitian inner product $\bgm{\cdot}{\cdot}$ and let \[
	\f{u}(\f g,J,\bgm{\cdot}{\cdot}) := \operatorname{Lie}(\mathsf{U}(\f g,J,\bgm{\cdot}{\cdot})) = \{A \in \f{gl}(\f g,J) : \bgm{A\cdot}{\cdot } + \bgm{\cdot}{A\cdot} = 0\}
	\] be its Lie algebra. For each $A \in \f{gl}(\f g,J)$ and $\mu \in \mathsf{GL}(\f g,J)\cdot [\cdot,\cdot]$, define 
	\[
		\pi(A)\mu := \mathrm{d}(h \mapsto h\cdot \mu)_{\id}A = A\mu - \mu(A\cdot,\cdot)-\mu(\cdot,A\cdot) \in T_\mu\mathsf{GL}(\f g,J)\cdot [\cdot,\cdot].
	\]
	
	Finally, given a complex Lie algebra $(\f g,\mu,J,g)$ with Lie bracket $\mu \in \mathcal{C}(\f g,J)$ and $J$-Hermitian inner product $g \in \operatorname{Sym}^2_+(\f g,J)$, define $P_\mu^g \in \f{gl}(\f g,J)$ to be the $g$-Hermitian endomorphism satisfying $g(P_\mu^g\cdot,\cdot) := \Theta_\mu(g)$, where $\Theta_\mu(g)$ is the torsion-twisted Chern--Ricci form associated to the simply-connected complex Lie group with left-invariant metric defined by the data $(\f g,\mu,J,g)$. Equation (\ref{eqn:Theta}) implies that for any left-invariant unitary frame $\{Z_i\}_{i=1}^n \subset \f g^{1,0}$,
	\begin{equation}
		\label{eqn:TTCR}
	P^g_\mu = \frac{1}{2} \sum_{i,j} g(\cdot,\mu(\B{Z_i},\B{Z_j}))\mu(Z_i,Z_j) = \frac{1}{2}\sum_{i} \operatorname{ad}_{Z_i}\circ \operatorname{ad}_{Z_i}^*,
	\end{equation}
	where $(\cdot)^*$ denotes the adjoint with respect to the Hermitian inner product $g$. When there is no confusion about the inner product being considered, we will write $P_{\mu} := P_\mu^g$. 
 	
 	The group $\mathsf{GL}(\f g,J)$ also acts naturally on the space $\f {gl}(\f g,J)$ via the adjoint action given by $h\cdot A = hAh^{-1}$ for all $h\in \mathsf{GL}(\f g,J), A \in \f{gl}(\f g,J)$. Biholomorphism invariance of $\Theta$ yields the following Lemma:
 	\begin{lem}
 		\label{lem:PEquiv}
 		For all $h\in \mathsf{GL}(\f g,J)$, $\mu \in \mathcal{C}(\f g,J)$ and $g \in \operatorname{Sym}^2_+(\f g,J)$, it holds that $h \cdot P_{\mu}^g = P^{h\cdot g}_{h\cdot \mu}$.
 	\end{lem}
 	
	\subsection{$\HCFp$-solitons}
	
	An $\HCFp$-soliton is a Hermitian manifold $(M,J,g)$ satisfying
	\[
	\Theta(g) = \lambda g + \mathcal{L}_Zg,
	\]
	for some $\lambda \in \R$ and holomorphic vector field $Z \in \Gamma(M,T^{1,0}M)$. In this case, there exists a solution $(g_t)_{t\in I}$ to the $\HCFp$ (\ref{eqn:hcfp}) with $g_0 = g$ given by $g_t = (1-\lambda t)\varphi_t^*g$ for some family of biholomorphisms $(\varphi_t)_{t\in I}$. Solitons therefore correspond to self-similar solutions of the $\HCFp$.
	
	On a Lie group, we can consider solitons with more structure. First, define
	\[
		\mathsf{Aut}(\f g,J,\mu) := \{h\in \mathsf{GL}(\f g,J) : h\cdot \mu = \mu\},
	\]
	as the Lie group of Lie algebra automorphisms of $(\f g,J)$. Let \[
		\operatorname{Der}(\f g,J,\mu) := T_{\id}\mathsf{Aut}(\f g,J,\mu) = \{D\in\f{gl}(\f g,J) : \pi(D)\mu = 0\},
	\]
	be its Lie algebra. When the bracket and complex structure are clear, we will write $\operatorname{Der}(\f g) = \operatorname{Der}(\f g,\mu) = \operatorname{Der}(\f g,\mu,J)$. A simply-connected complex Lie group $(\mathsf{G},J,g)$ with left-invariant Hermitian metric $g$ is called a \emph{semi-algebraic} soliton if
	\[
	\Theta(g) = \lambda g + \frac{1}{2}\left(g(D\cdot,\cdot) + g(\cdot , D\cdot)\right),
	\]
	for some $D \in \operatorname{Der}(\f g)$. The left-invariant solution corresponding to an algebraic soliton is a soliton in the usual sense where the biholomorphisms driving the evolution are in addition Lie group automorphisms of $(\mathsf{G},J)$.
	Equivalently, $(\mathsf{G},J,g)$ is a semi-algebraic soliton if its infinitesimal data $(\f g,J,\mu,g)$ satisfies
	\[
	P_\mu^g = \lambda \id_{\f g} + \frac{1}{2}(D + D^*),
	\]
	for some $D\in \operatorname{Der}(\f g,\mu)$, where $(\cdot)^*$ denotes the $g$-adjoint. Finally, a semi-algebraic soliton is called \emph{algebraic} if $D^* \in \operatorname{Der}(\f g,\mu)$.
	
%
%
%
	\begin{ex}\label{ex:heisenberg}
		Consider the \emph{complex Heisenberg group}:
		\[
				\mathsf{H}_{2n +1}(\C) := \begin{pmatrix}
				1&(\C^n)^*&\C\\0&\id_n&\C^n\\
				0&0&1
				\end{pmatrix}.
			\]
		Any left-invariant metric $g$ on $\mathsf{H}_{2n+1}(\C)$ is an algebraic soliton. To see this, let $\f{h}_{2n+1} := \operatorname{Lie}(H_{2n+1}(\C))$. Then there exists $Z \in \f{h}_{2n+1}$ such that $[\f h_{2n+1},\f h_{2n+1}] = \C Z$ and $[Z,\f h_{2n+1}] = 0$. Let $V = (\C Z)^{\perp_{g_e}}$. Then for all $X\in \f h_{2n+1}$, $g(P_{[\cdot,\cdot]}^gX,X) = C|g(X ,Z)|^2$. Thus, $P_{[\cdot,\cdot]}^g = C\id_{\C Z} = -C\id_{\f h_{2n+1}} + C(2\id_{\C Z} + \id_V)$. It is straight forward to verify that $2\id_{\C Z} + \id_V$ is a derivation of $\f{h}_{2n+1}$.
	\end{ex}

	A Hermitian metric $g$ is called $\HCFp$-\emph{static} if
	\[
	\Theta(g) = \lambda g.
	\]
	for some $\lambda \in \R$. Such metrics are the $\HCFp$ equivalent of Einstein metrics in the Ricci flow setting, in that they correspond to solutions that evolve only by scaling.

	The following necessary condition for the existence of left-invariant $\HCFp$-static metrics on complex Lie groups is known.
	
	\begin{prop}[{\cite[~Proposition 2.5]{stanfield21}}]
		\label{thm:staticImpPerf}
		Any complex Lie group admitting a left-invariant $\HCFp$-static metric is perfect.
	\end{prop}
	Recall that a Lie group $\mathsf{G}$ is \emph{perfect} if $[\f g,\f g] = \f g$ and that all semisimple Lie groups are perfect. 
	
	We will now use the structure theory of complex semisimple Lie groups to construct a canonical class of left-invariant $\HCFp$-static metrics. To that end, recall a Cartan decomposition of a complex semisimple Lie group $(\mathsf{G},J)$ is given by $\mathsf{G} = \mathsf{K}\exp(J_e \f k)$, where $\mathsf{K}$ is any maximal compact subgroup of $\mathsf{G}$ and $\f k = T_e\mathsf{K}$ is its Lie algebra. At the level of Lie algebras, we have $\f g = \f k \oplus J_e \f k$. With this, we make the following Definition:
	
	\begin{defn}
		\label{def:canonicalMetric}
		Let $(\mathsf{G},J)$ be a semisimple, complex Lie group, $\mathsf{K}\leq \mathsf{G}$ a maximal compact subgroup, and $\kappa\colon \f k \times \f k \to \R$ the Killing form of $\f k$. Define $g^{\mathsf{K}}$ to be the unique left-invariant Hermitian metric on $(\mathsf{G},J)$ such that $g^{\mathsf{K}}_e$ is given by $g_e^{\mathsf{K}}|_{\f k \times \f k} = -\kappa$, and extended as a $J_e$-Hermitian inner product to $\f g = \f k \oplus J_e \f k$.
	\end{defn}
	Note that $g^{\mathsf{K}}_e$ is $\operatorname{Ad}(\mathsf{K})$-invariant. This definition is canonical in the sense that it defines a $\mathsf{G}$-equivariant-isometry class of Hermitian metrics which doesn't depend on the choice of maximal compact subgroup $\mathsf{K}$, as Proposition \ref{prop:canonicalMetric} now shows.
	\begin{prop}
		\label{prop:canonicalMetric}
		Suppose $\mathsf{K},\mathsf{L}\leq (\mathsf{G},J)$ are two maximal compact subgroups of a semisimple, complex Lie group. Then $g^{\mathsf{L}} = \phi^*g^{\mathsf{K}}$ for some inner automorphism $\phi\colon (\mathsf{G},J) \to (\mathsf{G},J)$.
	\end{prop}
	\begin{proof}
		By standard theory, $L$ is conjugate to $\mathsf{K}$ in $\mathsf{G}$, so $\f l = T_e\mathsf{L} = \operatorname{Ad}_x\f k$ for some $x \in \mathsf{G}$. For $\operatorname{Ad}_xX,\operatorname{Ad}_xY \in \f l$,
		\[
		\begin{split}
		g_e^{\mathsf{L}}(\operatorname{Ad}_xX,\operatorname{Ad}_xY) =&  -\tr(\ad_{\operatorname{Ad}_xX}\circ \ad_{\operatorname{Ad}_xY}|_{\f l})\\ =& -\tr(\operatorname{Ad}_{x^{-1}}\circ\ad_{\operatorname{Ad}_xX}\circ \ad_{\operatorname{Ad}_xY}\circ \operatorname{Ad}_x|_{\f k})\\
		=& -\tr(\ad_{X}\circ \ad_{Y}|_{\f k}) = g_e^{\mathsf{K}}(X,Y).
		\end{split} 
		\]
		Thus, since $J_e$ commutes with $\operatorname{Ad}_x$, $g_e^{\mathsf{L}}(\operatorname{Ad}_x\cdot,\operatorname{Ad}_x\cdot) = g_e^{\mathsf{K}}$ and the result follows.
	\end{proof}
	\begin{prop}[{\cite[~Theorem 5.5]{UstHCFHom2017}}]
		\label{prop:canonicalMetricStatic}
		Let $(\mathsf{G},J)$ be a semisimple, complex Lie group and $\mathsf{K} \leq \mathsf{G}$ a maximal compact subgroup. Then the metric $g^{\mathsf{K}}$ is $\HCFp$-static.
		\end{prop}
	\begin{ex}
		Let $(\mathsf{G},J) = \mathsf{SL}_n(\C)$, then $\mathsf{K} = \mathsf{SU}(n+1)$ is a maximal compact subgroup and $g^{\mathsf{SU}(n+1)}_A(X,Y) = 2(n+1)\tr(A^{-1}X(A^{-1}Y)^*)$, for all $A\in \mathsf{SL}_n(\C)$ and $X,Y \in T_A\mathsf{SL_n(\C)} \cong A\f{sl}_n(\C)$. Here $(\cdot)^*$ denotes the conjugate transpose.
	\end{ex}
	
	It would be interesting to find examples of perfect, non-semisimple complex Lie groups with $\HCFp$-static metrics.	
	\section{Positive Hermitian Curvature Flow on the special linear groups}
	\label{sec:SLn}
	In this section we study the behaviour of a three parameter family of left-invariant solutions to the $\HCFp$ on $\mathsf{SL}_{n+1}(\C)$. We will show Theorem \ref{thm:SLnUnstable} by constructing solutions that begin arbitrarily close to the canonical static metric $g^{\mathsf{SU}(n+1)}$, yet do not converge to it after appropriate rescaling. We will then analyse the true Cheeger--Gromov limit of such solutions. We remark the choice to study solutions on $\mathsf{SL}_n(\C)$ was made after computing the linearisation of the left-invariant $\HCFp$ on various complex Lie groups and finding a centre eigenspace on $\mathsf{SL}_n(\C)$. We study only the global dynamics in this article.
	
	\subsection{Torsion-twisted Chern--Ricci curvature of $\operatorname{Ad}(\mathsf{SU}(n))$-invariant metrics on $\mathsf{SL}_{n+1}(\C)$}~\\
	Let ${\mathsf{SL}_{n+1}(\C)}$ be the complex Lie group of invertible $(n+1)\times(n+1)$ matrices of unit determinant, and $\f{sl}_{n+1}(\C)$ its Lie algebra of $(n+1) \times (n+1)$ traceless matrices. Denote by $[\cdot,\cdot]$ the usual commutator bracket. Define $\bgm{X}{Y} := \tr(XY^*)$ for all $X,Y \in \f{sl}_{n+1}(\C)$. Up to linear isometry and scaling, this is the unique $\operatorname{Ad}(\mathsf{SU}(n+1))$-invariant Hermitian inner product on $\f {sl}_{n+1}(\C)$. In the notation of Section \ref{sec:prelims}, $g^{\operatorname{Ad}(\mathsf{SU}(n+1))}_e = 2(n+1)\bgm{\cdot}{\cdot}$. Let $e_{ij}$ be the standard matrix with the only non-zero entry being $1$ in the $(i,j)$ position.
	
	There is a natural embedding 
	\[
	\f{sl}_n(\C) \to \begin{pmatrix}
		\f{sl}_n(\C)&0\\
		0&0
	\end{pmatrix} \subset \f{sl}_{n+1}(\C),
	\]
	and so $\f{sl}_n(\C)$ acts naturally on $\f{sl}_{n+1}(\C)$ via the adjoint representation. It follows that
	\[
	\f{sl}_{n+1}(\C) = \f{sl}_n(\C) \oplus  \C I \oplus \f{s},
	\]
	decomposes into $\f{sl}_n(\C)$-invariant subspaces, where
	\begin{equation}
		\label{eqn:basis}
	I := \sqrt{\frac{1}{n(n+1)}}\begin{pmatrix}
		\id_n&0\\
		0&-n
	\end{pmatrix}, \qquad \f{s} := \langle r_i := e_{i(n+1)}, s_i := e_{(n+1)i}\rangle_{i=1}^n.
	\end{equation}
	Note that with respect to $\bgm{\cdot}{\cdot}$ this decomposition is unitary. Let $\{u_k\}_{k=1}^{n^2-1}$ be a unitary basis for $\f{sl}_n(\C)$ such that $u_k^* = \pm u_k$. Then $\{u_k\}_{k=1}^{n^2-1}\cup \{I\}\cup\{r_i\}_{i=1}^n \cup \{s_i\}_{i=1}^n$ is a unitary basis for $\f{sl}_{n+1}(\C)$. We note here that $\f s$ decomposes into two copies of the standard representation of $\f{sl}_{n}(\C)$, but this is not necessary to consider for our ansatz.
	
	Simple computations show that the (non-zero) brackets satisfy
	\begin{equation}
		\label{eqn:slnBracketRelations}
	[\f{sl}_n(\C),\f{sl}_n(\C)] = \f{sl}_n(\C), \quad [\f{sl}_n(\C),\f{s}] = \f{s},
	\end{equation}
\begin{equation}
	\label{eqn:otherBracketRelations}
	[I,\f s] = \f s, \quad [\f s,\f s] = \f{sl}_n(\C) \oplus \C I.
\end{equation}
	For each $x,y,z > 0$, define $\sigma_{x,y,z} := x^{-1}\id_{\f{sl}_n(\C)} + y^{-1}\id_{\C I} + z^{-1}\id_{\f s} \colon \f{sl}_{n+1}(\C) \to \f{sl}_{n+1}(\C)$. Consider the $\operatorname{Ad}(\mathsf{SU}(n))$-invariant inner product $\bgm{\sigma_{x,y,z}\cdot}{\cdot}$. We will compute the torsion-twisted Chern--Ricci operator $P_{x,y,z} := P^{\bgm{\sigma_{x,y,z} \cdot}{\cdot}}_{[\cdot,\cdot]}$.
	\begin{lem}\label{lem:P} For all $x,y,z > 0$,
		\[
		\begin{split}
		P_{x,y,z} =& (nx+\frac{z^2}{x})\id_{\f{sl}_n(\C)} + (n+1)\frac{z^2}{y}\id_{I} + \frac{n+1}{n}((n-1)x + y)\id_{\f{s}}.
		\end{split} 
		\]
	\end{lem}
	\begin{proof}
		The proof is inspired by that of Leite and Dotti De Miatello in \cite{Leite82}. Let $x,y,z > 0$ and $\sigma = \sigma_{x,y,z}$. A straightforward computation shows that for any $v \in \f{sl}_{n+1}(\C)$, the adjoint of $\ad_v$ with respect to the metric $\bgm{\sigma\cdot}{\cdot}$ is given by $\ad_v^* = \sigma^{-1}\ad_{v^*}\sigma$. Moreover, $\{x^{1/2}u_k\}_{k=1}^{n^2-1}\cup\{y^{1/2}I\}\cup\{z^{1/2}s_i\}_{i=1}^n$ is a $\bgm{\sigma \cdot}{\cdot}$-unitary basis of $\f{sl}_{n+1}(\C)$. By (\ref{eqn:slnBracketRelations}), $\ad_v$ commutes with $\sigma$ for all $v\in \f{sl}_n(\C)$. Since $u_k^* = \pm u_k$ and $r_i = s_i^*$, it follows that
		\[
		\begin{split}
			P_{x,y,z} =& \frac{1}{2}x \sum \ad_{u_k^*}\ad_{u_k} + \frac{1}{2}y\ad_I \sigma^{-1}\ad_I\sigma\\
			&+ \frac{1}{2}z\left(\sum \ad_{r_i^*}\sigma^{-1}\ad_{r_i}\sigma + \sum\ad_{s_i^*}\sigma^{-1}\ad_{s_i}\sigma\right).
		\end{split}
		\]
		Recall that the Killing form of $\f{sl}_{n+1}(\C)$ is given by $\tr(\ad_X \circ \ad_Y) = 2(n+1)\tr(XY) = 2(n+1)\bgm{X}{Y^*}$. Thus,
		\begin{equation}
			\label{eqn:SumOfadBasis}
		\sum\ad_{u_k^*}\ad_{u_k} + \ad_{I}^2 + \sum\ad_{r_i^*}\ad_{r_i} + \sum\ad_{s_i^*}\ad_{s_i} = 2(n+1)\id.
		\end{equation}
		Similarly, restricting to $\f{sl}_n(\C)$,
		\begin{equation}
			\sum\ad_{u_k^*}\ad_{u_k}\big|_{\f{sl}_n(\C)} = 2n\id_{\f{sl}_n(\C)}.
		\end{equation}
	Combining these yields
	\[
	\big(\sum \ad_{r_i^*}\ad_{r_i} + \sum \ad_{s_i^*}\ad_{s_i}\big)\big|_{\f{sl}_n(\C)} = 2\id_{\f{sl}_n(\C)}.
	\]
 	Thus, by (\ref{eqn:otherBracketRelations}) and the definition of $\sigma$, for $v\in \f{sl}_n(\C)$,
	\[
	\begin{split}
		P_{x,y,z}v = nxv + z^2x^{-1}v.
	\end{split}
	\]
	Similarly, by (\ref{eqn:SumOfadBasis}),
	\[
	\big(\sum \ad_{r_i^*}\ad_{r_i} + \sum \ad_{s_i^*}\ad_{s_i}\big)\big|_{\C I} = 2(n+1)\id_{\C I}.
	\]
	So, again by (\ref{eqn:otherBracketRelations}),
	\[
	P^{\bgm{\sigma \cdot}{\cdot}}I = (n+1)zI. 
	\]
	Finally we compute $P^{\bgm{\sigma \cdot}{\cdot}}$ on $\f{s}$. By direct computation, we have for all $1 \leq i \leq j$
	\[
	\ad_{I}r_i = \sqrt{\frac{(n+1)}{n}}r_i,\quad \ad_{I}s_i = -\sqrt{\frac{(n+1)}{n}}s_i.
	\]
	In particular, $\ad_I^2|_{\f{s}} = \frac{n+1}{n}\id_{\f{s}}$. For any $1 \leq i,j \leq n$,
	\begin{equation}
		\label{eqn:risjRels}
	\ad_{r_i}r_j = \ad_{s_i}s_j = 0,\qquad \ad_{r_i}s_j = e_{ij} - \frac{1}{n}\delta_{ij}\begin{pmatrix}\id_{\f{sl}_n(\C)}&0\\0&0\end{pmatrix}+ \delta_{ij}\sqrt{\frac{n+1}{n}}I.
	\end{equation}
	Thus,
	\begin{equation*}
		\begin{split}
	\ad_{r_i^*}\sigma^{-1}\ad_{r_i}\sigma s_j &= xz^{-1}\ad_{s_i}\left(e_{ij} - \frac{1}{n}\delta_{ij}\begin{pmatrix}\id_{\f{sl}_n(\C)}&0\\0&0\end{pmatrix}\right) + yz^{-1}\delta_{ij}\sqrt{\frac{n+1}{n}}\ad_{s_i}I\\
	&= xz^{-1}\left(s_j - \frac{1}{n}\delta_{ij}s_i\right) + yz^{-1}\frac{n+1}{n}\delta_{ij}s_i.\\
	\end{split}
	\end{equation*}
	Summing over $i$ yields
	\[
	\sum_{i}\ad_{r_i^*}\sigma^{-1}\ad_{r_i}\sigma s_j = z^{-1} \frac{n+1}{n}((n-1)x + y)s_j.
	\]
	A similar computation for $\ad_{s_i^*}\sigma^{-1}\ad_{s_i}r_j$ gives
	\[
	(\sum\ad_{r_i^*}\sigma^{-1}\ad_{r_i}\sigma + \sum\ad_{s_i}^*\sigma^{-1}\ad_{s_i}\sigma)|_{\f{s}} =  z^{-1} \frac{n+1}{n}((n-1)x + y)) \id_{\f{s}}.
	\]
	Using this, and (\ref{eqn:SumOfadBasis}) once gain, we see that $\sum\ad_{u_k^*}\ad_{u_k}|_{\f s} = \frac{n^2-1}{n}\id_{\f s}$. This yields \[P_{x,y,z}|_{\f{s}} = \frac{n+1}{n}((n-1)x + y)\id_{\f{s}}\] as required.
	\end{proof}

	\begin{cor}
		\label{cor:HCFpEquiv}
		Let $g$ be a left-invariant, Hermitian metric on ${\mathsf{SL}_{n+1}(\C)}$ induced by the inner product $\bgm{\sigma_{x_0,y_0,z_0}\cdot}{\cdot}$ on $\f{sl}_n(\C)$ for some $x_0,y_0,z_0 > 0$. Then the unique left-invariant $\HCFp$ solution $(g_t)_{t\in [0,T_{\max})}$ with $g_0 = g$ is induced by $\bgm{\sigma_{x,y,z}\cdot}{\cdot}$, where $(x,y,z)\colon [0,T_{\max}) \to \R_{>0}^3$ is a smooth solution to the system
	\begin{equation}
			\label{eqn:xyzODE}
		\begin{cases}
			\dot x = nx^2 + z^2,\quad &x(0)= x_0,\\
			\dot y = (n+1)z^2,\quad &y(0) = y_0,\\
			\dot z = \frac{n+1}{n}z((n-1)x + y),\quad &z(0) = z_0.
		\end{cases}
	\end{equation}
	\end{cor}
	\begin{proof}
		Suppose that $(x,y,z)$ is a smooth solution to (\ref{eqn:xyzODE}) and consider the family of inner products \[\bgm{\sigma_{x,y,z}\cdot}{\cdot} = x^{-1}\bgm{\cdot}{\cdot}|_{\f{sl}_n(\C)\times\f{sl}_n(\C)} + y^{-1}\bgm{\cdot}{\cdot}|_{\C I \times \C I} + z^{-1}\bgm{\cdot}{\cdot}|_{\f s \times \f s}.\]
		Then,
		\[
		\begin{split}
		\frac{\textrm{d}}{\textrm{d}t} \bgm{\sigma_{x,y,z}\cdot}{\cdot} =& -x^{-1}\left(nx+ \frac{z^2}{x}\right)\bgm{\cdot}{\cdot}|_{\f{sl}_n(\C)\times\f{sl}_n(\C)} -y^{-1}(n+1)\frac{z^2}{y}\bgm{\cdot}{\cdot}|_{\C I \times \C I} \\
		& - z^{-1}\frac{n+1}{n}((n-1)x+y)\bgm{\cdot}{\cdot}|_{\f s \times \f s}\\
		=& -\bgm{\sigma_{x,y,z}P_{x,y,z}\cdot}{\cdot} = -\Theta(\bgm{\sigma_{x,y,z}\cdot}{\cdot}),
		\end{split}
		\]
		and the result follows from uniqueness of ODE solutions.
	\end{proof}

	\subsection{$\HCFp$ of $\operatorname{Ad}(\mathsf{SU}(n))$-invariant metrics on $\mathsf{SL}_{n+1}(\C)$}
	Given the computations of the previous section, we now turn to the analysis of (\ref{eqn:xyzODE}).
	\subsubsection{Blow up}
	We first show that the ODE (\ref{eqn:xyzODE}) blows up in finite time for any initial condition $(x_0,y_0,z_0)\in \R_{>0}^3$.
	
	\begin{prop}
		\label{prop:finTimeEx}
		Let $(x,y,z)$ be a $C^1$ solution to (\ref{eqn:xyzODE}) on $[0,T)$, where $T$ is the maximal time of existence. Then $T \leq \left((n+1)\min\{x_0,y_0,z_0\}\right)^{-1} < \infty$. Moreover,
		\[
		 x,y,z \geq \frac{1}{\min\{x_0,y_0,z_0\}^{-1}-(n+1)t} ,\qquad \forall t \in [0,T).
		\]
	\end{prop}
	\begin{proof}
		Define $\alpha := \min\{x,y,z\}$ which is $C^1$ a.e. in $[0,T)$. It follows from (\ref{eqn:xyzODE}) that
		\[
			\dot \alpha \geq (n+1)\alpha^2, \quad \text{ a.e. in } [0,T).
		\]
		Thus, $\alpha > \left(\alpha_0^{-1} - (n+1)t\right)^{-1}$ by comparison with $f' = (n+1)f^2$, where $\alpha_0 := \min\{x_0,y_0,z_0\} > 0$. It follows that $T \leq \frac{(n+1)}{\alpha_0} < \infty$. The estimate also follows from the definition of $\alpha$.
	\end{proof}
	\subsubsection{Stability}
	Since solutions to (\ref{eqn:xyzODE}) blow up in finite time, we will now study the rescaled system to understand the asymptotic behaviour. In particular, we rescale so that $x \equiv 1$.
	\begin{prop}
		Let $(\tilde x,\tilde y,\tilde z)\in \R_{>0}^3$ be a solution to (\ref{eqn:xyzODE}). Then, after a re-parametrisation in time, $(y,z) = (\tilde y/\tilde x,\tilde z / \tilde x)$ solves
		\begin{equation}
			\label{eqn:yzODE}
			\begin{cases}
				\dot y = z^2(n+1 - y) - ny,\quad &y(0) = y_0 >0\\
				\dot z = \frac{(n+1)}{n}z(n-1 + y) - (n+z^2)z,\quad &z(0) = z_0 > 0.
			\end{cases}
		\end{equation}
	\end{prop}
	\begin{proof}
		By standard theory, after a time re-parametrisation, the rescaled system $(x,y,z) = (\tilde x/\tilde x = 1,\tilde y/\tilde x,\tilde z / \tilde x)$, is given by (\ref{eqn:xyzODE}) up to the addition of a scalar multiple (say $\lambda \in \R$) of the identity. Since $\frac{\mathrm{d}}{\mathrm{d}t} \frac{\tilde x}{\tilde x} = 0$, $\lambda = -n - z^2$. The result immediately follows.
	\end{proof}
	This system has two fixed points $(1,1)$ and $(0,0)$. The first corresponds to the canonical static metric $g^{\mathsf{SU}(n+1)}$ and the second is degenerate.

	Let $v\colon\R^2 \to \R^2$ be the vector field associated to the system (\ref{eqn:yzODE}). i.e.
	\[
		v(y,z) := (z^2(n+1 - y) - ny, \frac{(n+1)}{n}z(n-1 + y) - (n+z^2)z).
	\]
	We will now study the global stability of the fixed point $(0,0)$.
	\begin{thm}
		\label{thm:stability}
		The point $(0,0)$ is a global attractor for the system (\ref{eqn:yzODE}) on the set 
		\[
			D := \{(y,z)\in\R^2_{>0} : \frac{z^2(n+1)}{z^2 + n} \leq y < 1\}.
		\]
	\end{thm}
	\begin{proof}
		Let $(y,z) \in D$. We observe that $\bgm{(1,0)}{ v(y,z)} \leq 0$ if and only if $y \leq \frac{z^2(n+1)}{z^2 + n}$, and the same statement is true when replaced with equalities. The inward pointing normal to the boundary segment $\partial D_1 := \{(y,z)\in D : y = \frac{z^2(n+1)}{z^2 + n}\}\subset \partial D$ is given by $N_{y,z} = ((n+z^2)^2,-2n(n+1)z)$ for $(y,z) \in \partial D_1$. It follows that for all $(y,z) \in \partial D_1$,
		\[
		\bgm{N_{y,z}}{v(y,z)} = \frac{2 n (n+1) z^2 \left(z^2-1\right)^2}{n+z^2} \geq 0,
		\]
		with equality if and only if $z \in \{0,1,-1\}$. Thus, $v(y,z)\in \operatorname{int} \mathcal{T}_{(y,z)}D$, the interior of the tangent cone of $D$ at $(y,z)$, for all $(y,z)\in \partial D$. It follows that $D$ is an invariant subset for the ODE (\ref{eqn:yzODE}).
		
		Now let $(y,z)$ be a solution to (\ref{eqn:yzODE}) in $D$, which exists for all positive times as it remains in the compact set $\B D$. In this case, $\dot z \leq 0$ if and only if $y \leq \frac{1}{n+1}(nz^2 + 1)$. 
		
		Now if $(y,z)$ is a solution to (\ref{eqn:yzODE}) with $(y_0,z_0) \in D$, then it exists for all positive times as it remains in the compact subset $\overline D$. Thus, since $\dot y \leq 0$, the Poincar\'e--Bendixson Theorem implies that the omega limit must be $\{(0,0)\}$, consisting of the unique fixed point inside $D$. 
	\end{proof}
	\begin{rmk}
		Notice that $B_\varepsilon((1,1))\cap D \neq 0$ for all $\varepsilon > 0$, so there exist solutions starting arbitrarily close to the fixed point $(1,1)$ that converge to the fixed point $(0,0)$.
	\end{rmk}
	Since the solutions to (\ref{eqn:yzODE}) correspond to rescaling (\ref{eqn:xyzODE}) so that $x \equiv 1$, it is important to understand the precise asymptotic behaviour of $x$. To that end, we now state and prove the following Lemma:
	\begin{lem}
		\label{lem:xAsymp}
		Let $D$ be as in Theorem \ref{thm:stability} and $(x,y,z)$ a solution to (\ref{eqn:xyzODE}) with initial condition $(1,y_0,z_0) \in D$ on the interval $[0,T_{\max})$. Then $y,z < x$ and there exists a constant $C > 0$ such that 
		\[
		\frac{1}{C(T_{\max} - t)} < x < \frac{C}{T_{\max} - t},
		\]
		for all $t \in [0,T_{\max})$.
	\end{lem}
	\begin{proof}
		Since $D \subset (0,1)\times (0,1)$ is an invariant subset of (\ref{eqn:xyzODE}), we have that $y,z < x$ as claimed. It follows that $x \to \infty$ as $t \to T_{\max}$. Applying $z < x$ to equation (\ref{eqn:xyzODE}) yields $\dot x < (n+1)x^2$ which implies $x < (1-(n+1)t)^{-1}$ by comparison. In particular $T_{\max} \geq (n+1)^{-1}$. Clearly $x^{-1} \to 0$ as $t\to T_{\max}$. We also have that $(x^{-1})' = -n - \frac{z^2}{x^2} \to -n$ as $t\to T_{\max}$, since $\frac{z}{x} \to 0$ by Theorem \ref{thm:stability}. Thus, since $|(x^{-1})'|$ stays bounded and away from $0$ as $t\to T_{\max}$, there exists $C > 0$ such that $C^{-1}(T_{\max}-t) < x^{-1} < C(T_{\max}-t)$ as required.
	\end{proof}
	\subsubsection{Convergence to a non-trivial soliton}
	We now study the asymptotics of $\HCFp$ solutions given by solutions to (\ref{eqn:yzODE}) which converge to the origin. To do this, we start with a key Lemma.
	\begin{lem}\label{lem:asymyz}
		Suppose $(y,z)$ is a $C^1$ solution to (\ref{eqn:yzODE}) on $[0,\infty)$ such that $(y,z) \to (0,0)$. Then
		\[
			\frac{z^2}{y} \to \frac{n^2-2}{n(n+1)}.
		\]
		as $t \to \infty$.
	\end{lem}

	\begin{proof}
		From (\ref{eqn:yzODE}), the quantity $\frac{z^2}{y}$ satisfies
		\[
		\begin{split}
		\frac{\mathrm{d}}{\mathrm{d}t}\ln\left(\frac{z^2}{y}\right) 
		=& (n+1)\left(\frac{n^2-2}{n(n+1)}- \frac{z^2}{y}\right) + 2\frac{n+1}{n}y - z^2.
		\end{split}
		\]
		It follows that for all $t \geq 0$,
		\[
		(n+1)\left(\frac{n^2-2}{n(n+1)}- \frac{z^2}{y}\right) - z^2 \leq \ln\left(\frac{z^2}{y}\right)' \leq  (n+1)\left(\frac{n^2-2}{n(n+1)}- \frac{z^2}{y}\right) + 2\frac{n+1}{n}y.
		\]
		Fix $0 < \varepsilon < \frac{n^2-2}{n(n+1)}$ and define $D_\varepsilon = \{t \geq 0  : 2\frac{n+1}{n}|y| + |z|^2 < \varepsilon(n+1)\}$ which contains a subset of the form $[r,\infty)$ for some $r > 0$ since $y,z \to 0$. It then follows from the differential inequality that
		\[
		(n+1)\frac{z^2}{y}\left(\frac{n^2 - 2}{n(n+1)} - \varepsilon - \frac{z^2}{y}\right)\leq\left(\frac{z^2}{y}\right)' \leq 	(n+1)\frac{z^2}{y}\left(\frac{n^2 - 2}{n(n+1)} + \varepsilon - \frac{z^2}{y}\right)
		\]
		on $D_\varepsilon$. By comparison, and noting that $f = c > 0$ is a global attractor of the ODE $f' = (n+1)f(c-f)$ on $\R_{>0}$, we see that $\liminf_{t\to\infty}\frac{z^2}{y} \geq \frac{n^2-2}{n(n+1)} - \varepsilon$ and $\limsup_{t\to\infty}\frac{z^2}{y} \leq \frac{n^2-2}{n(n+1)} + \varepsilon$. Sending $\varepsilon \to 0$ gives the result.
		\end{proof}
	
		It will be helpful to consider the flow from the varying brackets point of view. To that end, for $y,z > 0$, let $h_{y,z} := \sigma^{1/2}_{1,y,z}$ and $\mu_{y,z} := h_{y,z}\cdot [\cdot,\cdot]$. It follows that
		\[
		h_{y,z}\colon  (\f g,[\cdot,\cdot], \bgm{\sigma_{1,y,z}\cdot}{\cdot} = \bgm{h_{y,z}\cdot}{h_{y,z}\cdot})\to (\f g, \mu_{y,z}, \bgm{\cdot}{\cdot})
		\]
		is an isomorphism of complex metric Lie algebras. Moreover, this implies that the unique simply-connected complex Lie groups with left-invariant metrics corresponding to $(\f g,[\cdot,\cdot], \bgm{\sigma_{1,y,z}\cdot}{\cdot})$ and $(\f g, \mu_{y,z}, \bgm{\cdot}{\cdot})$ respectively are biholomorphically, equivariantly isometric. Let $P_{y,z} := P_{\mu_{y,z}}^{\bgm{\cdot}{\cdot}}$ be the torsion-twisted Chern--Ricci operator associated to this latter group. By Lemma \ref{lem:PEquiv},
		\[
		P_{y,z} = h_{y,z}P^{\bgm{\sigma_{1,y,z}\cdot}{\cdot}}h_{y,z}^{-1} = P^{\bgm{\sigma_{1,y,z}\cdot}{\cdot}}_{[\cdot,\cdot]}.
		\]
		We can now compute the explicit form of $\mu_{y,z}$.
		\begin{prop}\label{prop:muInf}
			For all $y,z > 0$,
			\[
			\begin{split}
			\mu_{y,z} = [\cdot,\cdot]\Big|_{\f{sl}_n(\C)\wedge\f{sl}_n(\C)} &+ [\cdot,\cdot]\Big|_{\f{sl}_n(\C) \wedge \f s} + \sqrt{y}[\cdot,\cdot]\Big|_{\C I \wedge \f s} \\&+ z\Pr{}_{\f {sl}_n(\C)} \circ [\cdot,\cdot]\Big|_{\f s \wedge \f s}+ \frac{z}{\sqrt{y}} \Pr{}_{\C I}\circ[\cdot,\cdot]\Big|_{\f s \wedge \f s},
			\end{split}
			\]
			where $\Pr{}_V$ denotes orthogonal projection onto any subspace $V \subset \f g$.
		\end{prop}
		\begin{proof}
			Recalling that $h_{y,z} = \id_{\f{sl}_n(\C)} + y^{-1/2}\id_{\C I} + z^{-1/2}\id_{\f s}$, the result immediately follows from (\ref{eqn:slnBracketRelations}), (\ref{eqn:otherBracketRelations}), and the definition of $\mu_{y,z}$.
		\end{proof}
		\begin{cor}
			\label{cor:limitBracketSLn}
			Let $(y,z)$ be a solution to (\ref{eqn:yzODE}) for which $y,z \to 0$, then $\mu_{y,z} \in \f g \otimes \Lambda^2 \f g^*$ converges as $t \to \infty$ to a bracket $\mu_\infty \in \f g \otimes \Lambda^2 \f g^*$ given by
			\[
			\begin{split}
				\mu_\infty := 	 [\cdot,\cdot]\Big|_{\f{sl}_n(\C)\wedge\f{sl}_n(\C)} &+ [\cdot,\cdot]\Big|_{\f{sl}_n(\C) \wedge \f s} + \sqrt{\frac{n^2-2}{n(n+1)}} \Pr{}_{\C I} \circ[\cdot,\cdot]\Big|_{\f s \wedge \f s}.
			\end{split}
			\]
			Moreover, the torsion-twisted Chern-Ricci operator $P_\infty = \lim_{t\to\infty}P_{y,z}$ of $\mu_\infty$ is given by
			\[
			P_\infty = n\id - \frac{1}{n}D,
			\]
			where $D := 2\id_{\C I} +\id_{\f s} \in \operatorname{Der}(\f g,\mu_\infty)$ is a derivation of the Lie bracket $\mu_\infty$. In particular, the simply-connected complex Lie group with left-invariant Hermitian metric corresponding to $(\f g,\mu_\infty,\bgm{\cdot}{\cdot})$ is a shrinking algebraic $\HCFp$-soliton.
		\end{cor}
	\begin{proof}
		The claim that $\mu_{y,z} \to \mu_\infty$ is an immediate consequence of Proposition \ref{prop:muInf} and Lemma \ref{lem:asymyz}. By Lemma \ref{lem:P} and again Lemma \ref{lem:asymyz}, we see that
		\[
		P_\infty = n\id_{\f{sl}_n(\C)} + \frac{n^2 - 2}{n}\id_{\C I} + \frac{n^2-1}{n}\id_{\f s} = n\id - \frac{1}{n}D.
		\]
		Finally, it is a straightforward computation to check that $D \in \operatorname{Der}(\f g,\mu_\infty)$. 
	\end{proof}
	The next Lemma describes the structure of the Lie algebra $(\f g,\mu_\infty)$.
	\begin{lem}
		\label{lem:limitAlgebra}
		The Lie algebra $(\f g,\mu_\infty) = \f{sl}_n(\C) \ltimes \f{h}_{2n + 1}$ is given by a semi-direct product, where $\f{h}_{2n+1}$ is the complex Heisenberg Lie algebra. 
	\end{lem}
	\begin{proof}
		We will show that $(\C I \oplus \f s,\mu_\infty)$ is isomorphic to $\f h_{2n+1}$. Indeed, if $\{r_i,s_i\}_{i=1}^n$ denotes the basis defined in (\ref{eqn:basis}), then by (\ref{eqn:risjRels}), $
		\mu_{\infty}(r_i,s_j) = \frac{\sqrt{n^2-2}}{n}\delta_{ij}I$. Moreover, $\mu_{\infty}(I,\f s) = \{0\}$. These are precisely the bracket relations for the Heisenberg Lie algebra $\f{h}_{2n+1}$. 
	\end{proof}

	\begin{proof}[Proof of Theorem \ref{thm:SLnUnstable}]
		Let $n\geq 2$ and $\varepsilon > 0$. Consider $({\mathsf{SL}_n(\C)},g^{\mathsf{SU}(n+1)})$, where $g^{\mathsf{SU}(n+1)}$ is as in Definition \ref{def:canonicalMetric}. Note that in the notation of Section \ref{sec:intro}, $g^{\operatorname{can}} = g^{\mathsf{SU}(n+1)}$ up to pull-back by an inner automorphism. Then, $g^{\mathsf{SU}(n+1)}_e = 2(n+1)\bgm{\sigma_{1,1,1}\cdot}{\cdot}$. Take $(y_0,z_0) \in B_{\varepsilon}((1,1))\cap D$, where $D$ is as in Theorem \ref{thm:stability}. Let $(x,y,z)$ solve (\ref{eqn:xyzODE}) on the interval $[0,T_{\max})$ with initial condition $(1,y_0,z_0)$. The corresponding left-invariant $\HCFp$ solution is $g_t = \bgm{\sigma_{x,y,z}\cdot}{\cdot}$ by Corollary \ref{cor:HCFpEquiv}. By Corollary \ref{cor:limitBracketSLn}, $\mu_{\frac yx,\frac zx} \to \mu_{\infty}$ as $t\to T_{\max}$. Now, by Lemma \ref{lem:xAsymp}, $x \sim (T_{\max} - t)^{-1}$ as $t\to T_{\max}$. Hence, for any sequence of times $t_n \to T_{\max}$, after passing to a subsequence, $\mu_{y(t_n)(T_{\max} - t_n),z(t_n)(T_{\max}-t_n)} \to c \mu_\infty$ as $n\to \infty$ for some $c\in \R$. By making $\varepsilon$ arbitrarily small, the theorem follows from \cite[~Corollary 6.20]{lauretConvHomMfds} and Lemma \ref{lem:limitAlgebra} since $(T_{\max}-t_n)^{-1} g_{t_n} = \bgm{\sigma_{x(t_n)(T_{\max} - t_n),y(t_n)(T_{\max} - t_n),z(t_n)(T_{\max} - t_n)}\cdot}{\cdot}$.
	\end{proof}

	\section{Perfect solitons}
	\label{sec:perfSol}
	In this section, we construct new examples of $\HCFp$-solitons on simply-connected, perfect Lie groups. Let $(\f h,\mu, J_0,\bgm{\cdot}{\cdot}_0)$ be a complex Lie algebra with Lie bracket $\mu \in \f h \otimes \Lambda^2 \f h^*$, complex structure $J_0$ and $J_0-$Hermitian inner product $\bgm{\cdot}{\cdot}_0$ that is $\HCFp$-static. Without loss of generality, we rescale the inner-product so that $P^{\bgm{\cdot}{\cdot}_0} = \mu\mu^* = \id$. By Theorem \ref{thm:staticImpPerf}, the Lie algebra $(\f h,\mu)$ is perfect. That is, $\mu(\f h \wedge \f h ) = \f h$. Define
	\[
	(\f g,\nu) := (\f h,\mu)\ltimes(\f h,0),
	\]
	where the first factor acts via the adjoint representation on the second.  Explicitly, for $X,Y \in \f h$,
	\[
	\nu(X\oplus 0,Y\oplus 0) = \mu(X,Y)\oplus 0,\qquad \nu(X\oplus 0,0\oplus Y) = 0\oplus \mu(X,Y).
	\]
	We denote by $\bgm{\cdot}{\cdot}$ the unique inner product on $\f g$ induced by $\bgm{\cdot}{\cdot}_0$ making the two factors orthogonal and define $J := J_0 \oplus J_0$. Consider the embedding $\R^\times \ltimes \R \subset \mathsf{GL}(\f g,J)$ given by
	\[
	(a,b)\mapsto h_{a,b} := \begin{pmatrix}a&b\\0&1\end{pmatrix},
	\]
	where the blocks correspond to the splitting $\f h \oplus \f h$ and by scalar entries we mean scalar multiples of the identity. Define $\nu_{a,b} := h_{a,b}\cdot \nu \in  \R^\times \ltimes \R\cdot \nu$. Then for $X,Y \in \f h$,
	\[
	\begin{split}
		\nu_{a,b}(X\oplus0,Y\oplus 0) =& a^{-1}\mu(X,Y)\oplus 0,\\ \nu_{a,b}(X\oplus 0,0\oplus Y) =& 0\oplus a^{-1}\mu(X,Y), \text{ and} \\ \nu_{a,b}(0\oplus X,0\oplus Y) =& -\frac{b^2}{2a}\mu(X,Y)\oplus \frac{-2b}{a}\mu(X,Y).
		\end{split}
	\]
	We can now compute the torsion-twisted Chern--Ricci operator $P_{\nu_{a,b}}$.
	\begin{prop}
		For all $(a,b) \in \R^\times \ltimes \R$,
		\[
		P_{\nu_{a,b}} = \begin{pmatrix}a^{-2}+ 2 a^{-2}b^4&4a^{-2}b^3\\4a^{-2}b^3&2a^{-2}+8a^{-2}b^2\end{pmatrix}
		\]
		
	\end{prop}
	\begin{proof}
		Let $X = X_1 \oplus X_2 \in \f h \oplus \f h$. Let $\{e_i\}_{i=1}^n$ be a unitary basis of $\f h$. Then, using the structure of $\nu_{a,b}$ and the fact that $\mu\mu^* = \id$,
		\[
			\begin{split}
				\bgm{P_{\nu_{a,b}}X}{X} =& \sum_{i<j}|\bgm{\nu_{a,b}(e_i\oplus 0,e_j\oplus 0)}{X}|^2 + \sum_{i,j}|\bgm{\nu_{a,b}(e_i\oplus 0,0\oplus e_j)}{X}|^2 \\
				&+\sum_{i<j}|\bgm{\nu_{a,b}(0\oplus e_i,0\oplus e_j)}{X}|^2\\
				=& a^{-2}\|X_1\|^2 + 2a^{-2}\|X_2\|^2 + \sum_{i,j} \left|\frac{b^2}{a}\bgm{\mu(e_i,e_j)}{X_1}+ \frac{2b}{a}\bgm{\mu(e_i,e_j)}{X_2}\right|^2\\
				=& a^{-2}\|X_1\|^2 + 2a^{-2}\|X_2\|^2 + 2 a^{-2}b^4\|X_1\|^2\\ &+ 4a^{-2}b^3(\bgm{X_1}{X_2}+\bgm{X_2}{X_1}) + 8a^{-2}b^2\|X_2\|^2\\
				=& \left\langle\begin{pmatrix}a^{-2}+ 2 a^{-2}b^4&4a^{-2}b^3\\4a^{-2}b^3&2a^{-2}+8b^2a^{-2}\end{pmatrix}X,X\right\rangle,
			\end{split}
		\]
		and the result follows since $P_{\nu_{a,b}}$ is $\bgm{\cdot}{\cdot}$-Hermitian.
	\end{proof}
	We now investigate the existence of solitons in this orbit. Note that $P_{\nu_{a,b}} = a^{-2}P_{\nu_{1,b}}$, so it suffices to study $P_{\nu_t}$, where $\nu_t := \nu_{1,t}$.
%
	\begin{thm}
		\label{thm:PerfAlsSols}
		The following hold:
		\begin{enumerate}
			\item $\nu_t$ is a semi-algebraic soliton  bracket for $t \in \{0,2^{-1/4},-2^{-1/4}\}$ and $\nu_0$ is algebraic.
			\item $\nu_{\pm 2^{-1/4}}$ are not algebraic.
			\item $\nu_{-2^{-1/4}} \in \mathsf{U}(\f g,J,\bgm{\cdot}{\cdot})\cdot \nu_{2^{-1/4}}$.
		\end{enumerate}
	\end{thm}
	\begin{proof}
		It is straightforward to check that, $D := \begin{pmatrix}0&0\\0&1\end{pmatrix}$ is a derivation of $\nu$.  Thus, since $\operatorname{Der}(\f g,h\cdot\nu) = h\operatorname{Der}(\f g,\nu)h^{-1}$, if we define $h_t := h_{1,t}$, then
		\[
		D_t := h_tDh_t^{-1} = \begin{pmatrix}0&t\\0&1\end{pmatrix}
		\]
		is a derivation of $\nu_t = h_t \cdot \nu$ for all $t \in \R$. We can then write
		\[
		\begin{split}
			P_{\nu_t} =& \begin{pmatrix}1+2t^4&0\\0&2\end{pmatrix} + 4t^2\bigg(\begin{pmatrix}0&t\\0&1\end{pmatrix} + \begin{pmatrix}0&0\\t&1\end{pmatrix}\bigg)\\
			=& \begin{pmatrix}1+2t^4&0\\0&2\end{pmatrix} + 2t^2(D_t + D_t^*).
		\end{split}
		\]
		Thus, $P_{\nu_0} = \id + D$. Moreover, $P_{\nu_{\pm 2^{-1/4}}} = 2\id + D_{\pm 2^{-1/4}}+D_{\pm2^{-1/4}}^*$, so $\nu_{2^{-1/4}}$ and $\nu_{-2^{-1/4}}$ are both semi-algebraic solitons, so (1) follows. To see that they are not algebraic, it suffices to show that $D_t^*$ is not a derivation of $\nu_t$ for $t \neq 0$. To this end, observe that
		\[
		h_t^{-1}D_t^*h_t = \begin{pmatrix}1&-t\\0&1\end{pmatrix}\begin{pmatrix}0&0\\t&1\end{pmatrix}\begin{pmatrix}1&t\\0&1\end{pmatrix} = \begin{pmatrix}-t^2&-t(1+t^2)\\t&1+t^2\end{pmatrix},
		\]
		which does not preserve the nilradical of $\nu$ unless $t = 0$. So $D_t^*$ is not a derivation of $\nu_t$ for $t\neq 0$, which proves (2). To see (3), define $k := (1) \oplus (-1)$, which one can check lies in $\mathsf{U}(\f g,J,\bgm{\cdot}{\cdot})\cap \operatorname{Aut}(\f g,\nu)$, then $k\cdot \nu_{-1} = (kh_{-1})\cdot \nu = (h_1k)\cdot \nu = h_1\cdot \nu = \nu_1,$
		as required.
	\end{proof}
	For each $t$, let $(\mathsf{G}_t,g_t)$ be the unique simply-connected complex Lie group with left-invariant Hermitian metric corresponding to the infinitesimal data $(\f g,\nu_t,J,\bgm{\cdot}{\cdot})$. Clearly, $\mathsf{G}_t$ is isomorphic to $\mathsf{G}_\nu$, the simply-connected complex Lie group corresponding to $(\f g,\nu)$ for all $t\in\R$. The next result demonstrates non-uniqueness of solitons on $\mathsf{G}_\nu$.
	\begin{prop}
		\label{prop:distinctSols}
		$(\mathsf{G}_0,g_0)$ and $(\mathsf{G}_1,g_1)$ are not equivalent up to homothety.
	\end{prop}
	\begin{proof}
		Suppose $f\colon (\mathsf{G}_0,g_0) \to (\mathsf{G}_1,g_1)$ is a biholomorphism such that $f^*g_1 = \lambda^2 g_0$ for some $\lambda \in \R$. We can assume $f(e) = e$, since otherwise we can replace $f$ with $L_{f(e)^{-1}}\circ f$. Then biholomorphism invariance of the torsion twisted Chern--Ricci tensor implies $\lambda^{-2}P_{\nu_0} = \mathrm{d}f_e^{-1} P_{\nu_1}\mathrm{d}f_e$. Thus, $\lambda^{-2}\tr P_{\nu_0} = \tr P_{\nu_1}$ and $\lambda^{-2\dim \mathsf{G}_0}\det(P_{\nu_0}) = \det(P_{\nu_1})$. But $\tr P_{\nu_0} = 3$, $\tr P_{\nu_1} = 8$, $\det P_{\nu_0} = 2$ and $\det P_{\nu_1} = 8$. It follows that $\frac{8}{3} = (\frac{8}{2})^{\dim \mathsf{G}_0}$, which is a contradiction.
	\end{proof}
	Theorem \ref{thm:TwoSolitons} now follows from Theorem \ref{thm:PerfAlsSols} and Proposition \ref{prop:distinctSols}.
	
	It is natural to consider the stability of the solitons corresponding to $\nu_0$ and $\nu_{\pm 2^{-1/4}}$ under the $\HCFp$. Indeed, analysis of the \emph{gauged bracket flow} as constructed in \cite{arrLafHomPlu19} and \cite{bohmLafImmRicFl} on the orbit $\R^\times \ltimes \R \cdot \nu$ yields that the soliton corresponding to $\nu_0$ is unstable, and the one corresponding to $\nu_{\pm 2^{-1/4}}$ is an attractor for the (suitably normalised) $\HCFp$ restricted to metrics on the orbit $\R^\times\ltimes \R\cdot \bgm{\cdot}{\cdot} = \{h_{a,b}\cdot \bgm{\cdot}{\cdot}: (a,b)\in \R^\times \ltimes \R\}$.

\bibliography{Bibliography.bib}

\begin{thebibliography}{10}

\bibitem{arrLafHomPlu19}
Romina~M. Arroyo and Ramiro~A. Lafuente.
\newblock The long-time behavior of the homogeneous pluriclosed flow.
\newblock {\em Proc. Lond. Math. Soc. (3)}, 119(1):266--289, 2019.

\bibitem{bohmLafImmRicFl}
Christoph B\"{o}hm and Ramiro~A. Lafuente.
\newblock Immortal homogeneous {R}icci flows.
\newblock {\em Invent. Math.}, 212(2):461--529, 2018.

\bibitem{BohmWilking2Pos2008}
Christoph B\"{o}hm and Burkhard Wilking.
\newblock Manifolds with positive curvature operators are space forms.
\newblock {\em Ann. of Math. (2)}, 167(3):1079--1097, 2008.

\bibitem{bolingPCFLocHom}
Jess Boling.
\newblock Homogeneous solutions of pluriclosed flow on closed complex surfaces.
\newblock {\em J. Geom. Anal.}, 26(3):2130--2154, 2016.

\bibitem{BrendleSchoenQuart2009}
Simon Brendle and Richard Schoen.
\newblock Manifolds with {$1/4$}-pinched curvature are space forms.
\newblock {\em J. Amer. Math. Soc.}, 22(1):287--307, 2009.

\bibitem{KRS}
Xiuxiong Chen, Song Sun, and Gang Tian.
\newblock A note on {K}\"{a}hler-{R}icci soliton.
\newblock {\em Int. Math. Res. Not. IMRN}, (17):3328--3336, 2009.

\bibitem{pluriclosedTwoStep2015}
Nicola Enrietti, Anna Fino, and Luigi Vezzoni.
\newblock The pluriclosed flow on nilmanifolds and tamed symplectic forms.
\newblock {\em J. Geom. Anal.}, 25(2):883--909, 2015.

\bibitem{fino2021pluriclosed}
Anna Fino, Nicoletta Tardini, and Luigi Vezzoni.
\newblock Pluriclosed and strominger k$\backslash$" ahler-like metrics
  compatible with abelian complex structures.
\newblock {\em arXiv preprint arXiv:2102.01920}, 2021.

\bibitem{HamiltonThreeManifolds1982}
Richard~S. Hamilton.
\newblock Three-manifolds with positive {R}icci curvature.
\newblock {\em J. Differential Geometry}, 17(2):255--306, 1982.

\bibitem{HamiltonFourManifolds1986}
Richard~S. Hamilton.
\newblock Four-manifolds with positive curvature operator.
\newblock {\em J. Differential Geom.}, 24(2):153--179, 1986.

\bibitem{jabRicSolAlg14}
Michael Jablonski.
\newblock Homogeneous {R}icci solitons are algebraic.
\newblock {\em Geom. Topol.}, 18(4):2477--2486, 2014.

\bibitem{HCFUni2019}
Ramiro~A. Lafuente, Mattia Pujia, and Luigi Vezzoni.
\newblock Hermitian curvature flow on unimodular {L}ie groups and static
  invariant metrics.
\newblock {\em Trans. Amer. Math. Soc.}, 373(6):3967--3993, 2020.

\bibitem{lauretConvHomMfds}
Jorge Lauret.
\newblock Convergence of homogeneous manifolds.
\newblock {\em J. Lond. Math. Soc. (2)}, 86(3):701--727, 2012.

\bibitem{lauret2015CurvatureFlows}
Jorge Lauret.
\newblock Curvature flows for almost-hermitian {L}ie groups.
\newblock {\em Trans. Amer. Math. Soc.}, 367(10):7453--7480, 2015.

\bibitem{lauretLaplacian}
Jorge Lauret.
\newblock Laplacian flow of homogeneous {$G_2$}-structures and its solitons.
\newblock {\em Proc. Lond. Math. Soc. (3)}, 114(3):527--560, 2017.

\bibitem{Leite82}
M.~L. Leite and I.~Dotti de~Miatello.
\newblock Metrics of negative {R}icci curvature on {${\rm SL}(n,\,{\bf R}),$}
  {$n\geq 3$}.
\newblock {\em J. Differential Geometry}, 17(4):635--641 (1983), 1982.

\bibitem{MoriProj1979}
Shigefumi Mori.
\newblock Projective manifolds with ample tangent bundles.
\newblock {\em Ann. of Math. (2)}, 110(3):593--606, 1979.

\bibitem{PanelliPodestaHCFCmpctHom2020}
Francesco Panelli and Fabio Podest\`a.
\newblock Hermitian {C}urvature {F}low on {C}ompact {H}omogeneous {S}paces.
\newblock {\em J. Geom. Anal.}, 30(4):4193--4210, 2020.

\bibitem{Pediconi2020}
Francesco Pediconi and Mattia Pujia.
\newblock Hermitian curvature flow on complex locally homogeneous surfaces.
\newblock {\em Annali di Matematica Pura ed Applicata (1923 -)}, Jul 2020.

\bibitem{PujiaSolitons2019}
Mattia Pujia.
\newblock Expanding solitons to the {H}ermitian curvature flow on complex {L}ie
  groups.
\newblock {\em Differential Geom. Appl.}, 64:201--216, 2019.

\bibitem{pujia2020positive}
Mattia Pujia.
\newblock Positive hermitian curvature flow on complex 2-step nilpotent lie
  groups.
\newblock {\em Manuscripta Mathematica}, 2020.

\bibitem{stanfield21}
James Stanfield.
\newblock Positive {H}ermitian curvature flow on nilpotent and almost-abelian
  complex {L}ie groups.
\newblock {\em Ann. Global Anal. Geom.}, 60(2):401--429, 2021.

\bibitem{StreetsAndTianHCF2011}
Jeffrey Streets and Gang Tian.
\newblock Hermitian curvature flow.
\newblock {\em J. Eur. Math. Soc. (JEMS)}, 13(3):601--634, 2011.

\bibitem{UstHCFHom2017}
Yury Ustinovskiy.
\newblock Hermitian curvature flow on complex homogeneous manifolds.
\newblock {\em Ann. Scuola Norm. Sup. Pisa Cl. Sci}, 06 2017.

\bibitem{ustinovskiyThesis}
Yury Ustinovskiy.
\newblock {\em Hermitian {C}urvature {F}low and {C}urvature {P}ositivity
  {C}onditions}.
\newblock ProQuest LLC, Ann Arbor, MI, 2018.
\newblock Thesis (Ph.D.)--Princeton University.

\bibitem{UstPosHCF2019}
Yury Ustinovskiy.
\newblock The {H}ermitian curvature flow on manifolds with non-negative
  {G}riffiths curvature.
\newblock {\em Amer. J. Math.}, 141(6):1751--1775, 2019.

\bibitem{UstStruct2020}
Yury Ustinovskiy.
\newblock On the structure of {H}ermitian manifolds with semipositive
  {G}riffiths curvature.
\newblock {\em Trans. Amer. Math. Soc.}, 373(8):5333--5350, 2020.

\end{thebibliography}
\bibliographystyle{plain}

\end{document}